\documentclass[12pt,oneside,reqno]{amsart}
\usepackage{}
\usepackage{amssymb}

\usepackage{bbm}
\usepackage{cases}
\usepackage{amsmath}
\usepackage{graphicx}
\usepackage{mathrsfs}
\usepackage{stmaryrd}
\usepackage{color}
\usepackage{soul}
\usepackage[dvipsnames]{xcolor}
\usepackage{amsfonts}
\usepackage{enumerate,amsmath,amssymb,amsthm}
\numberwithin{equation}{section}

\newcommand{\be}{\begin{eqnarray}}
\newcommand{\mE}{\end{eqnarray}}
\newcommand{\ce}{\begin{eqnarray*}}
\newcommand{\de}{\end{eqnarray*}}
\newtheorem{theorem}{Theorem}[section]
\newtheorem{lemma}[theorem]{Lemma}
\newtheorem{remark}[theorem]{Remark}
\newtheorem{definition}[theorem]{Definition}
\newtheorem{proposition}[theorem]{Proposition}
\newtheorem{example}[theorem]{Example}
\newtheorem{corollary}[theorem]{Corollary}

\def\eps{\varepsilon}

\def\p{\partial}

\def\[{{\Big[}}
\def\]{{\Big]}}
\def\<{{\langle}}
\def\>{{\rangle}}
\def\({{\Big(}}
\def\){{\Big)}}

\def\bx{{\mathbf{x}}}

\def\dif{{\mathord{{\rm d}}}}

\def\no{\nonumber}
\def\={&\!\!=\!\!&}
\def\bt{\begin{theorem}}
\def\et{\end{theorem}}
\def\bl{\begin{lemma}}
\def\el{\end{lemma}}
\def\br{\begin{remark}}
\def\er{\end{remark}}

\def\bd{\begin{definition}}
\def\ed{\end{definition}}
\def\bp{\begin{proposition}}
\def\ep{\end{proposition}}
\def\bc{\begin{corollary}}
\def\ec{\end{corollary}}
\def\bx{\begin{example}}
\def\ex{\end{example}}

\def\cF{{\mathcal F}}

\def\cI{{\mathcal I}}

\def\cM{{\mathcal M}}

\def\mE{{\mathbb E}}

\def\mH{{\mathbb H}}
\def\mI{{\mathbb I}}

\def\mL{{\mathbb L}}

\def\mN{{\mathbb N}}

\def\mP{{\mathbb P}}

\def\mR{{\mathbb R}}

\def\sF{{\mathscr F}}

\def\sH{{\mathscr H}}

\def\sL{{\mathscr L}}

\def\geq{\geqslant}
\def\leq{\leqslant}

\allowdisplaybreaks

\begin{document}
	\title{Strong solutions of stochastic differential equations with coefficients in mixed-norm spaces}
	
	\date{}

	\author{Chengcheng Ling \,\, and\,\, Longjie Xie}
	
	\address{Chengcheng Ling: School of science, Beijing Jiaotong University, Beijing 100044, China\\
    Fakult\"at f\"ur Mathematik, Universit\"at Bielefeld, Bielefeld,  33615, Germany\\
    Email: chengchenglcc@126.com
	}

	\address{Longjie Xie:
		School of Mathematics and Statistics, Jiangsu Normal University,
		Xuzhou, Jiangsu 221000, P.R.China\\
		Email: longjiexie@jsnu.edu.cn
	}

	\thanks{Research of C. Ling is  supported  by the DFG through the IRTG 2235. L. Xie is supported by the Alexander-von-Humboldt foundation and NNSF of China (No. 11701233, 11931004).}
	
	\begin{abstract}
		By studying parabolic equations in mixed-norm spaces, we  prove the existence and uniqueness of strong solutions
		to stochastic differential  equations driven by Brownian motion with coefficients in spaces with mixed-norm, which extends Krylov and
		R\"ockner's result in \cite{Kr-Ro} and Zhang's result in \cite{Zh1}.
		
		\bigskip
		
		
	\end{abstract}
	
	\maketitle

\section{Introduction and main result}

Consider the following stochastic differential equation (SDE for short):
\begin{align}\label{sde00}
\dif X_t=b(t,X_t)\dif t+\dif W_t,\quad X_0=x\in \mR^d,
\end{align}
where $d\geq1$, $b: \mR_+\times\mR^d\rightarrow\mR^d$ is a Borel measurable function, and $(W_t)_{t\geq0}$ is a standard Brownian motion defined on some filtered  probability space $(\Omega,\sF,(\sF_t)_{t\geq 0},\mP)$. It is a classical result that if the coefficient $b$ is global Lipschitz continuous in $x$ uniformly with respect to $t$, then there exists a unique strong solution $(X_t(x))_{t\geq0}$ to SDE (\ref{sde00}) for every $x\in\mR^d$. However, many important applications of this class of SDE show that the Lipschitz continuity imposed on the coefficient is a rather severe restriction. Thus a lot of attentions have been paid to seek a strong solution for (\ref{sde00}) under weaker assumptions on the drift $b$.
A remarkable result due to Zvonkin \cite{Z} showed that if $d=1$ and $b$ is bounded, then SDE (\ref{sde00}) admits a unique strong solution for each
$x\in\mR$. Zvonkin's result was then extended to the multidimensional case by Veretennikov \cite{Ve}.
A further generalization was obtained by Krylov and R\"ockner \cite{Kr-Ro} where the pathwise uniqueness  for SDE (\ref{sde00}) was shown when
\begin{align}
b\in L^q_{loc}(\mR_+;L^p_{loc}(\mR^d))\quad\text{with}\quad q,p\in(2,\infty)\quad\text{and}\quad  d/p+2/q<1.    \label{pq}
\end{align}
Later, Zhang \cite{Zh1} extends these results to SDE driven by multiplicative noise
\begin{align}\label{sde1}
\dif X_t=b(t,X_t)\dif t+\sigma(t,X_t)\dif W_t,\quad X_0=x\in \mR^d
\end{align}
under the assumptions that $\sigma:\mR_+\times\mR^d\rightarrow\mR^d\otimes\mR^d$ is a bounded, uniformly elliptic matrix-valued function which is uniformly continuous in $x$ locally uniformly with respect to $t$, and
$$
b, |\nabla\sigma|\in L^q_{loc}(\mR_+;L^p_{loc}(\mR^d))
$$
with $p,q$ satisfying (\ref{pq}). Here and below, $\nabla$ denotes the weak derivative with respect to the  $x$ variable. Note that when $\sigma\equiv0$, SDE (\ref{sde1}) is just an ordinary differential equation, which is far from being well-posed under the above conditions on the drift coefficient. This reflects that noises can play some regularization effects to the deterministic systems, we refer the readers to \cite{F,Gu} for more comprehensive overview. From then on, there are also increasing interests of studying the  properties of the unique strong solution to SDE (\ref{sde1}) with singular coefficients, see e.g. \cite{Fe-Fl-1, M-N-P-Z,Wa,XZ2,Zh4} and the references therein.

\vspace{2mm}
To the best of our knowledge, the conditions imposed on the coefficients  in \cite{Kr-Ro,Zh1} are known to be the weakest  so far in the literature to ensure the strong well-posedness of the SDE (\ref{sde00}) and (\ref{sde1}).  However, there still exists a gap: the assumptions are not well-unified. Let us specify this point by a simple example. Consider the following SDE  in $\mR^{2}$:
\begin{equation} \label{sde0}
\left\{ \begin{aligned}
&\dif X^1_t=b_1(t,X^1_t)\dif t+\dif W^1_t,\qquad X^1_0=x_1\in\mR,\\
&\dif X_t^2=b_2(t,X^2_t)\dif t+\dif W^2_t,\qquad X^2_0=x_2\in\mR.
\end{aligned} \right.
\end{equation}
If we denote by $x:=(x_1,x_2)^T\in\mR^{2}$,
$X_t:=(X^1_t,X^2_t)^T$, $W_t:=(W^1_t,W^2_t)^T$,
and define the vector field
\begin{equation*}
b(t,x):=\left(
\begin{array}{c}
b_1(t,x_1) \\
b_2(t,x_2) \\
\end{array}
\right) : \mR_+\times\mR^2\rightarrow\mR^2.
\end{equation*}
Then SDE (\ref{sde0}) can be rewritten as
\begin{align}
\dif X_t=b(t,X_t)\dif t+\dif W_t,\quad X_0=x\in\mR^{2}.    \label{sde2}
\end{align}
According to the above mentioned condition (\ref{pq}), we need to assume
$$
b\in L^q_{loc}(\mR_+;L^p_{loc}(\mR^2))\quad\text{with}\quad 2/p+2/q<1
$$
to ensure the strong well-posedness of the SDE (\ref{sde2}). This in particular means that we need
\begin{align}\label{pq0}
b_1, b_2\in L^q_{loc}(\mR_+;L^p_{loc}(\mR^1))\quad\text{with}\quad 2/p+2/q<1.
\end{align}
On the other hand, the two-dimensional SDE (\ref{sde0}) can also be viewed as two independent equations for $X_t^1$ and $X_t^2$, since $X_t^1$ and $X_t^2$ are not involved with each other in the equation. From this point of view, SDE (\ref{sde0}) can be strongly well-posed under the condition that
\begin{align}\label{pq20}
b_1, b_2\in L^q_{loc}(\mR_+;L^p_{loc}(\mR^1))\quad\text{with}\quad 1/p+2/q<1,
\end{align}
which does not coincide with  (\ref{pq0}). We point out that such problem will always  appear when we consider SDEs in multi-dimension, and especially for degenerate noise cases and multi-scale models involving at least slow and fast phase variables, see e.g. \cite{F-F-P-V,RSX,X,Zh}.

\vspace{2mm}
The main aim of this work is to get rid of the above unreasonableness  by studying SDE (\ref{sde1}) with coefficients in  general mixed-norm spaces. To this end, let ${\bf p}=(p_1, \cdots ,p_d)\in[1,\infty)^d$ be a multi-index, we denote by $L^{\bf p}(\mathbb{R}^{d})$ the space of all measurable functions on $\mathbb{R}^{d}$ with norm
\begin{align*}
\|f\|_{L^{\bf p}(\mathbb{R}^d)}:=\bigg(\int_{\mR}\cdots\bigg(\int_{\mR}|f(x_1,...,x_d)|^{p_1}\dif x_1\bigg)^{\frac{p_2}{p_1}}\dif x_2\bigg)^{\frac{p_3}{p_2}}\cdots\dif x_d\bigg)^{\frac{1}{p_d}}<\infty.
\end{align*}
When $p_i=\infty$ for some $i=1,\cdots,d$, the norm is taken as the supreme over $\mathbb{R}$ with respect to the corresponding variable $x_i\in\mathbb{R}$, and by $L^{\bf p}_{loc}(\mathbb{R}^{d})$ we mean the corresponding local space defined as usual.
Notice that the order is important when taking the above integrals. If we permute the $p_i$s, then increasing the order of  $p_i$ gives the smallest norm, while decreasing the order gives the largest norm.

\vspace{2mm}
Our  main result in this paper is as follows.
\bt\label{main}
Assume that for some $p_1, \cdots, p_d, q\in(2,\infty]$ and every $T>0$,
\begin{align}\label{con}
|b|,|\nabla \sigma|\in L^q([0,T];L^{\bf p}_{loc}(\mR^d))\quad \text{with}\quad\frac{2}{q}+\frac{1}{p_1}+\cdots+\frac{1}{p_d}<1,
\end{align}
and for every $n\in\mN$, $\sigma$ is uniformly continuous in $x\in B_n:=\{x\in\mR^d: |x|\leq n\}$ uniformly with respect to $t\in[0,n]$, and there exist positive constants $\delta_n$ such that for all $(t,x)\in[0,n]\times B_n$,
\begin{align*}
|\sigma(t,x)\xi|^2\geq\delta_n|\xi|^2,\quad\forall\xi\in\mR^{d}.
\end{align*}
Then, for each $x\in\mR^d$ there exists a unique strong solution $X_t(x)$ up to an explosion time $\zeta(x)$ to SDE (\ref{sde1}) such that
$$\lim_{t\uparrow\zeta(x)}X_t(x)=+\infty,\quad a.s..$$
\et

\br
i) The advantage of (\ref{con}) lies in the flexible integrability of the coefficients.  More precisely, it allows the integrability of the coefficients to be small in some directions by taking the integrability index large for the other directions (not as functions of the whole space variable). With this condition, the problem of the tricky example mentioned before does  not appear  since we can take  another index to be $\infty$.  That is, according to Theorem \ref{main}, SDE \eqref{sde0} will be strongly well-posed if $b_1,b_2\in L^q([0,T],L^{\bf p}_{loc}(\mathbb{R}^2))$  with $$
{\bf p}=(p_1,p_2)\quad{\text{satisfying}}\quad 2/q+1/p_1+1/p_2<1.
$$  Since $b_1$ does not depend on $x_2$ and $b_2$ does not depend on $x_1$, we can take $p_2=\infty$ for $b_1$ and $p_1=\infty$ for $b_2$. Thus, the integrability conditions for $b_1$ and $b_2$ are $2/q+1/p_1<1$ and $2/q+1/p_2<1$ respectively, which coincides with \eqref{pq20} when we treat the SDE \eqref{sde0} as two independent equations.

\vspace{1mm}
ii)
As mentioned in \cite{Kr2}, the necessity of mixed-norm spaces arises
when the physical processes have different behavior with respect to each component. In view of (\ref{con}), it reflects the classical fact that the integrability of time variable and space variable has the ratio 1:2. Meanwhile, the integrability of each component of the space variable is the same, which is natural because the noise is non-degenerate. Such kind of mixed-norm spaces will be more important when studying SDEs with degenerate noises. This will be our future works.
\er

In a recent work \cite{RSX}, where we study the averaging principle for slow-fast SDEs with singular coefficients, Theorem \ref{main}  will play an important role in deriving the optimal conditions on the coefficients. Now, let us introduce the proof briefly. The key tool to prove our main result is the
$\mL^q_{\bf p}$-maximal regularity estimate for the following second order parabolic PDEs on $[0,T]\times\mR^d$:
\begin{align}
\p_tu(t,x)=\sL^{a}_2u(t,x)+\sL_1^bu(t,x)+f(t,x),\quad u(0,x)=0,      \label{pde1}
\end{align}
where  $\sL_2^{a}+\sL_1^{b}$ is the infinitesimal generator corresponding to SDE (\ref{sde1}), i.e.,
\begin{align*}
\sL_2^{a}u(t,x):=\tfrac{1}{2}a^{ij}(t,x)\p_{ij}u(t,x),\quad\sL_1^bu(t,x):=b^i(t,x)\p_iu(t,x)
\end{align*}
with $a(t,x)=(a^{ij}(t,x)):=(\sigma\sigma^T)(t,x)$, and $\p_i$ denotes the $i$-th partial derivative respect to $x$. Here we use Einstein's convention that the repeated indices in a product will be summed automatically.
To be more specific, for any $q\in(1,\infty)$ and ${\bf p}=(p_1, \cdots ,p_d)\in(1,\infty)^d$, we need  to establish the following estimate:
\begin{align}\label{KD1}
\|\nabla^2 u\|_{\mL^q_{\bf p}(T)}\leq C\|f\|_{\mL^q_{\bf p}(T)},
\end{align}
see Section 2 for the precise definition of $\mL^q_{\bf p}(T)$. Notice that when $p_1=\cdots=p_d=q$, it is a standard procedure to prove \eqref{KD1} by the classical freezing coefficient argument (cf. \cite{Zh4}).
However, for general $q\in(1,\infty)$ and ${\bf p}\in(1,\infty)^d$ it seems to be non-trival. When $a^{ij}$ is independent of $x$ and $p_1=\cdots=p_d$, estimate \eqref{KD1} was first proved by Krylov in \cite{Kr2}.
In the spatial dependent diffusion coefficient case, Kim \cite{Ky} showed
\eqref{KD1} only for $p_1=\cdots=p_d\leq q$. This was recently generalized to  $p_1=\cdots=p_d>1$ and $q>1$ in \cite{XXZZ} by a duality method. We shall further develop the argument used in \cite{XXZZ}, and combing with the interpolation technique, to prove that \eqref{KD1}   holds for  mixed-norms even in the space variable.
The main result is provided by Theorem \ref{mp}, which should be of independent interest in the theory of PDEs.

\vspace{2mm}

This paper is organized as follows: In Section 2, we study the maximal regularity estimate for second order parabolic equations. In Section 3, we prove our main theorem.
Throughout this paper, we use the following convention: $C$ with or without subscripts will denote a positive constant, whose value may change from one appearance to another, and whose dependence on parameters can be traced from calculations.

\section{Parabolic equations in mixed-norm spaces}

Fix $T>0$ and let $\mR^{d+1}_T:=[0,T]\times\mR^d$. This section is devoted to study the  parabolic equation (\ref{pde1}) on $\mR^{d+1}_T$ in general mixed-norm spaces.
We first introduce some  notations.
For any $\alpha\in\mR$ and ${\bf p}=(p_1, ... ,p_d)\in[1,\infty)^d$, let $H_{\bf p}^\alpha(\mathbb{R}^d):=(1-\Delta)^{-\alpha/2}\big(L^{\bf p}(\mR^d)\big)$
be the usual Bessel potential space with norm
$$
\|f\|_{H_{\bf p}^\alpha(\mathbb{R}^d)}:=\|(\mI-\Delta)^{\alpha/2}f\|_{L^{\bf p}(\mathbb{R}^d)},
$$
where  $(\mI-\Delta)^{\alpha/2}f$ is defined through Fourier's  transform
$$
(\mI-\Delta)^{\alpha/2}f:=\cF^{-1}\big((1+|\cdot|^2)^{\alpha/2}\cF f\big).
$$
Notice that for $n\in\mN$ and ${\bf p}=(p_1, ... ,p_d)\in[1,\infty)^d$, an equivalent norm in $H_{\bf p}^n(\mR^d)$ is given by
$$
\|f\|_{H_{\bf p}^n(\mathbb{R}^d)}=\|f\|_{L^{\bf p}(\mathbb{R}^d)}+\|\nabla^n f\|_{L^{\bf p}(\mathbb{R}^d)}.
$$
For $q\in[1,\infty)$ and any $S<T$, we denote $ \mL^q_{\bf p}(S,T):=L^q([S,T];L^{\bf p}(\mR^d))$. For simplicity, we will write $\mL^q_{\bf p}(T):= \mL^q_{\bf p}(0,T)$, and $\mL^\infty(T)$ consists of functions satisfying
$$
\|f\|_{\mL^\infty(T)}:=\sup_{t\in[0,T]}\sup_{x\in\mR^d}|f(t,x)|<+\infty.
$$
We also introduce  that for $\alpha\in\mR$,
$$
\mH_{\alpha,{\bf p}}^q(T):=L^q\big([0,T];H^{\alpha}_{\bf p}(\mR^d)\big),
$$
and the space $\sH_{\alpha,{\bf p}}^q(T)$ consists of the functions $u=u(t)$ on $[0,T]$ with values in the space of distributions on $\mR^d$ such that $u\in \mH_{\alpha,{\bf p}}^q(T)$ and $\p_t u\in \mL^q_{\bf p}(T)$.

Throughout this section, we always assume that

\vspace{2mm}
\noindent{\bf (H$a$):} $a(t,x)=(\sigma\sigma^T)(t,x)$ is uniformly continuous in $x\in\mR^d$ locally uniformly with respect to $t\in\mR_+$, and there exists a constant $\delta>1$ such that for all $\xi\in\mR^d$,
\begin{align}\label{a}
\delta^{-1}|\xi|^2\leq|a(t,x)\xi|^2\leq \delta|\xi|^2,\qquad\forall x\in\mR^d.
\end{align}

\vspace{1mm}

The main result of this section is as follows.

\bt\label{mp}
Assume that {\bf (H$a$)} holds, ${\bf p}\in (1,\infty)^d$ and $q\in(1,\infty)$.
Let $b\in \mL^{\tilde q}_{\bf \tilde p}(T)$ with ${\bf \tilde p},  \tilde q$  satisfying  $\tilde p_i\in[p_i,\infty)$, $\tilde q\in[q,\infty)$ for $1\leq i\leq d$ and $2/\tilde q+1/\tilde p_1+\cdots+1/\tilde p_d<1$.
Then for every $f\in \mL^q_{\bf p}(T)$,
there exists a unique solution $u\in \sH^{q}_{2,{\bf p}}(T)$ to equation (\ref{pde1}). Moreover,
we have the following estimates:
\begin{enumerate}[(i)]
	\item there is a constant $C_1=C(d,{\bf p},q,\|b\|_{\mL^{\tilde q}_{\bf \tilde p}(T)},T)>0$ such that
	\begin{align}\label{es1}
	\|\p_tu\|_{\mL^q_{\bf p}(T)}+\|u\|_{\mH_{2,{\bf p}}^q(T)}\leq C_1\|f\|_{\mL^q_{\bf p}(T)};
	\end{align}
	
	\item for any $\alpha\in [0,2-\tfrac{2}{q})$, there exists a  constant $C_T=C(d,{\bf p},q,\|b\|_{\mL^{\tilde q}_{\bf \tilde p}(T)},T)$ satisfying $\lim_{T\to0} C_T=0$  such that
		\begin{align}
	\Vert u\Vert_{\mathbb{H}_{\alpha,\bf p}^\infty(T)}\leq C_T\Vert f\Vert_{\mathbb{L}^{q}_{\bf p}(T)}.\label{esss2}
	\end{align}
\end{enumerate}
In particular, we have
	\begin{align}
	\|u\|_{\mL^\infty(T)}\leq \hat C_T\|f\|_{\mL^q_{\bf p}(T)}, \quad\text{if}\quad 2/q+1/p_1+\cdots+1/p_d<2,\label{es2}
	\end{align}
	and
	\begin{align}
	\|\nabla u\|_{\mL^\infty(T)}\leq \hat C_T\|f\|_{\mL^q_{\bf p}(T)},\quad\text{if}\quad 2/q+1/p_1+\cdots+1/p_d<1,   \label{es3}
	\end{align}
where $\hat C_T>0$ is a constant satisfying $\lim_{T\to0} \hat C_T=0$.
\et

We shall provide the proof of the above result in the following subsections.

\subsection{Smooth diffusion coefficients without drift.} In this subsection, we consider PDE (\ref{pde1}) on $\mR_T^{d+1}$ with $b\equiv0$, i.e.,
\begin{align}
\p_tu(t,x)-\sL_2^au(t,x)-f(t,x)=0,\quad u(0,x)=0.      \label{pde3}
\end{align}
We shall focus on the $\mL^q_{\bf p}$-maximal regularity a priori estimate for (\ref{pde3}). To this end, we assume that $a$ is smooth enough, i.e.,  $a$ satisfies {\bf (H$a$)} and for all $m\in\mN$,
$$
\|\nabla^m a^{ij}(t,\cdot)\|_\infty<\infty.
$$
Motivated by \cite{XXZZ}, we also need to consider the dual equation for (\ref{pde3}):
\begin{align}\label{PDE12}
\p_tw(t,x)+\tfrac{1}{2}\p_{ij}\big((a^{i j}(t,x)w(t,x)\big)+f(t,x)=0,\ \ w(T,x)=0.
\end{align}
Our aim in this subsection is to prove the following result.

\bt\label{pri}
For any ${\bf p}\in (1,\infty)^d$ and $q\in(1,\infty)$, there is a constant $C>0$  depending only on $d,{\bf p},q$, $T$  and the continuity modulus of $a$ such that
for every $f\in C^\infty_0([0,T]\times\mR^d)$,
\begin{align}\label{Eq3}
\|\nabla^2 u\|_{\mL_{\bf p}^q(T)}\leq C\|f\|_{\mL_{\bf p}^q(T)},\ \ \| w\|_{\mL_{\bf p}^{q}(T)}\leq C \|f\|_{\mH_{-2,\bf p}^{q}(T)},
\end{align}
where $u$ and $w$ are solutions of \eqref{pde3} and \eqref{PDE12} respectively.
Moreover,
for any $\alpha\in[0,2-\frac{2}{q})$, we have
\begin{align}\label{t}
\Vert u\Vert_{\mathbb{H}_{\alpha,\bf p}^\infty(T)}\leq C_T\Vert f\Vert_{\mathbb{L}^{q}_{\bf p}(T)},\quad\Vert w\Vert_{\mathbb{H}_{\alpha-2,\bf p}^\infty(T)}\leq C_T\Vert f\Vert_{\mathbb{H}^{q}_{-2,\bf p}(T)}.
\end{align}
where $C_T>0$ is a constant satisfying $\lim_{T\to0} C_T=0$.
\et

Before giving the proof of the above theorem,  we first show the following lemma for later use, which generalizes \cite[Lemma 2.6]{Kr2} (see also \cite[Lemma 3.5]{Ky}).

\bl\label{pe}
	Let $T\in[0,\infty)$, $p\in(1,\infty)$ and $n\in\mathbb{N}$. For $k=1,\cdots,n$, let $a_k:\mathbb{R_+}\rightarrow \mathbb{R}^d\otimes\mathbb{R}^d$ be measurable functions and there exists a constant $\delta\geq 1$ such that for all $t\in[0,T]$,
 $$\delta^{-1}|\xi|^2\leq a_k^{ij}(t)\xi_i\xi_j\leq \delta|\xi|^2,\quad \forall\xi\in\mathbb{R}^d.$$
Let $\lambda_k\in(0,\infty)$, $\gamma_k\in\mathbb{R}$ and $u_k\in\sH_{\gamma_k+2, p}^{p}(T)$ be the solution to the equation
	\begin{align*}
	\partial_tu^k=a_k^{ij}\partial_{ij}u^k+f^k,\quad u^k(0,x)=0
	\end{align*}
	with $f\in \mathbb{H}_{\gamma_k, p}^{p}(T)$. Denote by $\Lambda_k=(\lambda_k-\Delta)^{\gamma_k/2}$. Then for any $i=2,\cdots,d$, we have
	\begin{align}\label{estmix1}
	\int_0^T\!\!&\int_{\mathbb{R}}\cdots\int_{\mathbb{R}}\prod_{k=1}^n\Vert\Lambda_k\Delta u^k(t,\cdot,x_{i},\cdots, x_{d})\Vert_{L^p(\mathbb{R}^{i-1})}^p\dif x_i\cdots \dif x_{d} \dif t\nonumber\\&\leq C_0 \sum_{k=1}^n\int_0^T\!\int_{\mathbb{R}}\cdots\int_{\mathbb{R}}\Vert \Lambda_kf^k(t,\cdot,x_{i},\cdots, x_{d})\Vert_{L^p(\mathbb{R}^{i-1})}^p
	\nonumber\\&\quad\quad\quad\quad\quad\quad\quad\quad\times\prod_{j\neq k}\Vert\Lambda_j\Delta u^j(t,\cdot,x_{i},\cdots, x_{d})\Vert_{L^p(\mathbb{R}^{i-1})}^p\dif x_i\cdots \dif x_{d} \dif t,
	\end{align}
	where $C_0$ is a positive constant.
\el

\begin{proof}
	Without loss of generality we may assume $\gamma_k=0$. Define $v^k:=\Delta u^k$. For fixed $i=2,\cdots,d$ and $x_i,x_{i+1},\cdots,x_d\in\mR$,  take $X=(x^1,\cdots,x^n)$ with $x^j=(x^j_1,\cdots,x_d^j)\in\mathbb{R}^d$ such that  for  all $1\leq j\leq n$, $x^j_i\equiv x_i\in\mathbb{R}$, $x^j_{i+1}\equiv x_{i+1}\in\mathbb{R}$, $\cdots$, $x^j_{d}\equiv x_{d}\in\mathbb{R}$.  Hence, $X\in \mathbb{R}^{d+(n-1)(i-1)}$. For such $X$, we define
	$$V(t,X):=v^1(t,x^1)\times\cdots\times v^n(t,x^n).$$
	Then one can check that
	$$\partial_tV(t,X):=\mathbb{P}V(t,X)+F(t,X),$$
	where $$\mathbb{P}V=a_k^{ij}\frac{\partial^2V}{\partial x^k_i\partial x^k_j},$$
	$$F(t,X):=\Delta_{x^j}G^j(t,X), \quad G^j(t,X)=f^j(t,x^j)\prod_{j\neq k}v^k(t,x^k).$$
	By classical result (cf. \cite[Lemma 1.5]{Kr2}) we have
	$$\Vert V\Vert_{L^p([0,T]\times\mathbb{R}^{d+(n-1)(i-1)})}\leq C_0\sum_j\Vert G^j\Vert_{L^p((0,T)\times\mathbb{R}^{d+(n-1)(i-1)}},$$
	which is exactly \eqref{estmix1}. The lemma is proved.
\end{proof}

With the above preparation, we can give:

\begin{proof}[Proof of Theorem \ref{pri}]
Let ${\bf p}=(p_1,p_2,\cdots,p_d)\in(1,\infty)^d$ and $q\in(1,\infty)$. We divide  the proof into five steps: we first prove estimate \eqref{Eq3} in step 1-4, and in the fifth step we show estimate \eqref{t}.

	\vspace{1mm}	
\noindent\emph{Step 1.}  [Case $p_1=\cdots=p_d\in(1,\infty)$ and $q\in(1,\infty)$]. In this case, the estimate (\ref{Eq3}) was proved by \cite[Theorem 3.3]{XXZZ}.
	
	\vspace{1mm}
\noindent\emph{Step 2.}  [Case $p_1=\cdots=p_{d-1}\in(1,\infty)$ and $p_d=q\in (1,\infty)$].
We only prove the estimate for $w$ since the estimate for $u$  is similar and easier. By duality and the same argument as in the proof of \cite[Theorem 3.3]{XXZZ}, it is sufficient to prove the desired estimate when $q=p_d=np_{d-1}=\cdots=np_1=:np$ for $n\in\mathbb{N}_+$ and $p\in(1,\infty)$. That is to say, we shall prove:
	\begin{align*}
	\Vert  w\Vert_{L^{np}([0,T]\times\mathbb{R},{L^p(\mathbb{R}^{d-1})})}\leq C\Vert f\Vert_{\mathbb{H}^{np}_{-2,\bf p}(T)}^{np},\quad {\bf p}=(p,\cdots,p,np).
	\end{align*}
Take a non-negative smooth function $\phi$ supported in the ball $B_r:=\left\{x\in\mathbb{R}^{d}:|x|<r\right\}$ and $\int_{\mathbb{R}^d}|\phi|^{p}dx=1$, where $r$ is a small constant which will be determined below. For $x,z\in\mathbb{R}^d$, $s\in\mathbb{R}_+$, define $\phi_z(x):=\phi(x-z)$, $w_z(s,x):=w(s,x)\phi_z(x)$, $f_z(s,x):=f(s,x)\phi_z(x)$ and $a_z(s):=a(s,z)$. Then we can write
	\begin{align}\label{eqw}
	\partial_tw_z+\partial_{ij}(a^{ij}_zw_z)+g_z=0,\quad w_z(T,x)=0,
	\end{align}
	where $$g_z=f_z+\partial_{ij}(a^{ij}w)\phi_z-\partial_{ij}(a^{ij}_zw\phi_z).$$
	Below we  drop the time variable for simplicity, and for any $\gamma\in\mathbb{R}$ and fixed $x_d\in\mR$, we denote by $\Vert f(\cdot,x_d)\Vert_{H_p^\gamma(\mR^{d-1})}:=\Vert ((1-\Delta)^{\gamma/2}f)(\cdot, x_d)\Vert_{L^p(\mathbb{R}^{d-1})}$.  Notice that
	$$g_z=f\phi_z-2\partial_j(a^{ij}w)\partial_i\phi_z-a^{ij}w\partial_i\partial_j\phi_z+\partial_i\partial_j((a^{ij}-a^{ij}_z)w_z).$$
By the continuity of $a$, we have
	\begin{align*}
	\(\int_{\mathbb{R}^d}\Vert g_z(\cdot,x_d)\Vert_{H^{-2}_p(\mR^{d-1})}^p&\dif z
	\)^{1/p}\leq C_1\Vert f(\cdot,x_d)\Vert_{H^{-2}_p(\mR^{d-1})}\\
	&+C_{r}\sum_{i,j}\Vert (a^{ij}w)(\cdot,x_d)\Vert_{H^{-1}_p(\mR^{d-1})}\\
	&+C_r\sum_{i,j}\Vert a^{ij}w(\cdot,x_d)\Vert_{H^{-2}_p(\mR^{d-1})}+c_r\Vert w(\cdot,x_d)\Vert_{L^p(\mathbb{R}^{d-1})},
	\end{align*}
	where $C_r>0$ and $\lim_{r\rightarrow 0}c_r=0$.
	Let $\rho_n$ be a family of standard mollifiers and $a_n(t,x):=a(t,\cdot)*\rho_n(x)$ be the mollifying approximation of $a$. For every $\eps>0$, we can take $n$ large enough such that
	\begin{align*}
	\sum_{i,j}\Vert (a^{ij}w)(\cdot,x_d)&\Vert_{H^{-1}_p(\mR^{d-1})}+\sum_{i,j}\Vert a^{ij}w(\cdot,x_d)\Vert_{H^{-2}_p(\mR^{d-1})}\\\leq  & C_2\|(aw)(\cdot,x_d)\|_{H^{-1}_p(\mR^{d-1})}\\
	\leq& C_2\|(a_nw)(\cdot,x_d)\|_{H^{-1}_p(\mR^{d-1})}+C_2\|((a-a_n)w)(\cdot,x_d)\|_{H^{-1}_p(\mR^{d-1})}\\
	\leq& C_n\Vert w(\cdot,x_d)\Vert_{H^{-1}_p(\mR^{d-1})}+c_{1/n}\Vert w(\cdot,x_d)\Vert_{L^p(\mR^{d-1})}\\
	\leq &C_n\Vert w(\cdot,x_d)\Vert_{H^{-2}_p(\mR^{d-1})}+\eps\Vert w(\cdot,x_d)\Vert_{L^p(\mathbb{R}^{d-1})},
	\end{align*}
where the last step is due to the interpolation and Young's inequalities.
	Hence, we arrive at
		\begin{align}\label{g}
	\(\int_{\mathbb{R}^d}\Vert g_z(\cdot,x_d)\Vert_{H^{-2}_p(\mR^{d-1})}^p&\dif z
	\)^{1/p}\leq C_3\Vert f(\cdot,x_d)\Vert_{H^{-2}_p(\mR^{d-1})}\no\\
	&+ C_r\Vert w(\cdot,x_d)\Vert_{H^{-2}_p(\mR^{d-1})}+c_r\Vert w(\cdot,x_d)\Vert_{L^p(\mathbb{R}^{d-1})}.
	\end{align}
	Observe that
	\begin{align}\label{w}
	\Vert w\Vert_{L^{np}([0,T]\times\mathbb{R},{L^p(\mathbb{R}^{d-1})})}^{np}=&\int_0^T\!\!\!\int_{\mathbb{R}}\Big(\int_{\mathbb{R}^d}\Vert w(t,\cdot,x_d)\phi_z\Vert_{L^p(\mathbb{R}^{d-1})}^pdz\Big)^n\dif {x_d}\dif t\nonumber\\=&\int_0^T\!\!\!\int_{\mathbb{R}}\!\int_{\mathbb{R}^{nd}}\prod_{k=1}^n\Vert w_{z_k}(t,\cdot,x_d)\Vert_{{L^p(\mathbb{R}^{d-1})}}^p\dif {z_1}\cdots \dif z_{n}\dif {x_d}\dif t.
	\end{align}	
	Using Lemma \ref{pe},  we can deduce that
	\begin{align*}
	\int_0^T\!\!\int_{\mathbb{R}}&\prod_{k=1}^n\Vert w_{z_k}(t,\cdot,x_d)\Vert_{L^p(\mathbb{R}^{d-1})}^p\dif {x_d}\dif t\\&\leq C_4\sum_{k=1}^n\int_0^T\!\!\int_{\mathbb{R}}\Vert g_{z_k}(t,\cdot,x_d)\Vert_{H^{-2}_p(\mathbb{R}^{d-1})}^p\prod_{l\neq k}\Vert w_{z_l}(t,\cdot,x_d)\Vert_{L^p(\mathbb{R}^{d-1})}^p\dif x_{d}\dif t,
	\end{align*}
	which together with \eqref{g} and  \eqref{w}  implies
	\begin{align*}
	\Vert w\Vert_{L^{np}([0,T]\times\mathbb{R},L^p(\mathbb{R}^{d-1}))}^{np}&\leq C_5\sum_{k=1}^n\int_0^T\int_{\mathbb{R}}\int_{\mathbb{R}^{nd}}\Vert g_{z_k}(t,\cdot,x_d)\Vert_{H^{-2}_p(\mR^{d-1})}^p\\
	&\qquad\qquad\quad\quad\quad\times\prod_{l\neq k}\Vert w_{z_l}(t,\cdot,x_d)\Vert_{L^p(\mathbb{R}^{d-1})}^p\dif z_1\cdots \dif z_n\dif x_{d}\dif t
	\\ &=C_5n\int_0^T\int_{\mathbb{R}}\Big(\int_{\mathbb{R}^d}\Vert g_{z}(t,\cdot,x_d)\Vert_{H^{-2}_p(\mR^{d-1})}^p\dif z\Big)\\
	&\qquad\qquad\quad\times\Big(\int_{\mathbb{R}^d}\Vert  w_{z}(t,\cdot,x_d)\Vert_{L^p(\mathbb{R}^{d-1})}^p\dif z\Big)^{n-1}\dif x_d\dif t
	\\ &\leq C_6\int_0^T\int_{\mathbb{R}}\Big(\int_{\mathbb{R}^d}\Vert g_{z}(t,\cdot,x_d)\Vert_{H^{-2}_p(\mR^{d-1})}^pdz\Big)\\
	&\qquad\qquad\quad\qquad\quad\times\Vert w(t,\cdot,x_d)\Vert_{L^p(\mathbb{R}^{d-1})}^{(n-1)p}\dif x_d\dif t
	\\ &\leq C_6\int_0^T\int_{\mathbb{R}} \Vert f(t,\cdot,x_d)\Vert_{H^{-2}_p(\mR^{d-1})}^{np}\dif x_d\dif t\\
	&\qquad\qquad\quad\quad\quad+C_r\int_0^T\int_{\mathbb{R}}\Vert w(t,\cdot,x_d)\Vert_{H^{-2}_p(\mR^{d-1})}^{np}\dif x_d\dif t\\
	&\qquad\qquad\quad\quad\quad+c_r\Vert w\Vert_{L^{np}([0,T]\times\mathbb{R},L^p(\mathbb{R}^{d-1}))}^{np},
	\end{align*}
where the last inequality follows from H\"older's inequality and Young's inequality for product.
	 Let $r$ be small enough so that $c_r<1$, we can get that
	\begin{align}\label{wn}
	\Vert w\Vert_{L^{np}([0,T]\times\mathbb{R},L^p(\mathbb{R}^{d-1}))}^{np}&\leq C_7\bigg(\int_0^T\!\!\!\int_{\mathbb{R}} \Vert f(t,\cdot,x_d)\Vert_{H^{-2}_p(\mR^{d-1})}^{np}\dif x_d\dif t\nonumber\\&\quad+\int_0^T\!\!\!\int_{\mathbb{R}}\Vert w(t,\cdot,x_d)\Vert_{H^{-2}_p(\mR^{d-1})}^{np}\dif x_d\dif t\bigg).
	\end{align}
It remains to control the last term on the right hand side of the above inequality.	To this end, let $\kappa_{s,t}^z:=\int_s^ta_z(u)\dif u$ and
	$$P_{s,t}^z(x,x-y):=\frac{1}{\sqrt{(2\pi)^{d}\det(\kappa_{s,t}^z)}}e^{-\frac{(\kappa_{s,t}^z)^{-1}|x-y|^2}{2(t-s)}}.$$
	Then the solution of equation \eqref{eqw} is given by
	\begin{align*}
	w_z(t,x)=\int_t^T\!\!\!\int_{\mathbb{R}^d}P_{t,u}^z(x,x-y)g_z(u,y)\dif y\dif u.
	\end{align*} By {\bf (H$a$)} and  a standard interpolation technique, we get that for any $\alpha\in[0,2)$,
$$\Vert w_z(t,\cdot,x_d)\Vert_{H^{\alpha-2}_p(\mathbb{R}^{d-1})}\leq C_8\int_t^T(u-t)^{-\frac{\alpha}{2}}\Vert g_z(u,\cdot,x_d)\Vert_{H^{-2}_p(\mathbb{R}^{d-1})}\dif u.$$
Thus by Minkowski's inequality we have
\begin{align*}
\Vert w(t,\cdot,x_d)\Vert_{H^{\alpha-2}_p(\mathbb{R}^{d-1})}&\leq\Big(\int_{\mathbb{R}^d}\Vert w_z(t,\cdot,x_d)\Vert_{H^{\alpha-2}_p(\mathbb{R}^{d-1})}^p\dif z\Big)^{1/p}\\&\leq C_8\int_t^T(u-t)^{-\frac{\alpha}{2}}\Big(\int_{\mathbb{R}^d}\Vert g_z(u,\cdot,x_d)\Vert_{H^{-2}_p(\mathbb{R}^{d-1})}^p\dif z\Big)^{1/p}\dif u.
\end{align*}
Using \eqref{g} and the similar argument as in the proof of \eqref{wn}, we further have
\begin{align*}
\Vert w(t,\cdot,x_d)\Vert_{H^{\alpha-2}_p(\mathbb{R}^{d-1})}\leq C\int_t^T(u-t)^{-\frac{\alpha}{2}}\Big(&\Vert f(u,\cdot,x_d)\Vert_{H^{-2}_p(\mathbb{R}^{d-1})}\\
&+\Vert w(u,\cdot,x_d)\Vert_{H^{-2}_p(\mathbb{R}^{d-1})}\Big)\dif u.
\end{align*}
Let $\frac{1}{q'}+\frac{1}{np}=1$, then for any $\alpha\in[0,2-\frac{2}{np})$, we get by H\"older's inequality that
\begin{align}\label{5}
\Vert w(t,\cdot,x_d)&\Vert_{H^{\alpha-2}_p(\mathbb{R}^{d-1})}^{np} \leq C_9 \Big(\int_t^T(u-t)^{-\frac{q'\alpha}{2}}\dif u\Big)^{np/q'}\int_t^T\Big(\Vert f(u,\cdot,x_d)\Vert_{H^{-2}_p(\mathbb{R}^{d-1})}\nonumber\\&\quad\quad\quad\quad\quad\quad\quad\quad\quad\quad\quad\quad\quad\quad+\Vert w(u,\cdot,x_d)\Vert_{H^{-2}_p(\mathbb{R}^{d-1})}\Big)^{np}\dif u
\nonumber\\&\leq C_T\!\!\int_t^T\!\!\!\(\Vert f(u,\cdot,x_d)\Vert_{H^{-2}_p(\mathbb{R}^{d-1})}^{np}+\Vert w(u,\cdot,x_d)\Vert_{H^{-2}_p(\mathbb{R}^{d-1})}^{np}\)\dif u,
\end{align}
where $C_T>0$ satisfying $\lim_{T\to0} C_T=0$.
Then	by taking $\alpha=0$ and Gronwall's inequality we can obtain
	\begin{align}\label{sup}
	\sup_{s\in[0,T]}\Vert w(s,\cdot,x_d)\Vert_{H^{-2}_p(\mathbb{R}^{d-1})}^{np}\leq C_T \int_t^T\Vert  f(u,\cdot,x_d)\Vert_{H^{-2}_p(\mathbb{R}^{d-1})}^{np}\dif u,
	\end{align}
which in particular
implies that
	\begin{align*}
	\Vert w\Vert_{\mathbb{H}_{-2,\bf p}^\infty}^{np}\leq C_T\Vert f\Vert_{\mathbb{H}^{np}_{-2,\bf p}}^{np}.
	\end{align*}
  Taken this back into  \eqref{wn} yields that
\begin{align*}
	\Vert  w\Vert_{L^{np}([0,T]\times\mathbb{R},{L^p(\mathbb{R}^{d-1})})}\leq C_{10}\Vert f\Vert_{\mathbb{H}^{np}_{-2,\bf p}(T)}^{np},\quad {\bf p}=(p,\cdots,p,np).
	\end{align*}
	
 \vspace{1mm}\noindent
	\emph{Step 3.}  [Case $p_1=\cdots=p_{d-j}\in(1,\infty)$ and $p_{d-j+1}=\cdots =p_d=q\in (1,\infty)$ with any $1\leq j\leq d-1$].
	This can be proved  by following exactly the same arguments as in the proof of step 2, except that we need to use \eqref{estmix1}  Lemma \ref{pe} with $i=d-j+1$,  we omit the details.

	\vspace{1mm}\noindent
	\emph{Step 4.} [Interpolation] We develop an interpolation scheme to show the following claim:
	\begin{align}\label{indu}
{\text{ for every }} 1\leq j\leq d-1,
\eqref{Eq3} &{\text{ holds with }}	p_1=\cdots =p_{d-j}\in(1,\infty) \no\\
&{\text{ and }} p_{d-j+1}, p_{d-j+2},\cdots,p_d, q \in(1,\infty).
	\end{align}  In particular, when $j=d-1$, we get the desired result.

Interpolate the results in step 1 and step 2, we can get that \eqref{Eq3} holds when $p_1=\cdots =p_{d-1}\in(1,\infty)$ and $p_d, q\in(1,\infty)$. Thus, the assertion (\ref{indu}) is true for $j=1$. Assume that (\ref{indu}) holds for some $j=n-1\leq d-2$, we proceed to show that (\ref{indu}) is true for $n$.
For this, we  first  interpolate $p_1=\cdots =p_{d}\in(1,\infty)$ and $q\in(1,\infty)$ with $p_1=\cdots =p_{d-j}\in(1,\infty)$ and $p_{d-j+1}=\cdots=p_d=q\in(1,\infty)$ (both of which hold according to step 3) to get that the \eqref{Eq3} holds for $p_1=\cdots =p_{d-j}\in(1,\infty)$ and $p_{d-j+1}=p_{d-j+2}=\cdots=p_d,q\in(1,\infty)$. Then we interpolate $p_1=\cdots =p_{d-j}\in(1,\infty)$ and $p_{d-j+1}=p_{d-j+2}=\cdots=p_d,q\in(1,\infty)$ with $p_1=\cdots =p_{d-1}\in(1,\infty)$ and $p_d,q\in(1,\infty)$ (which holds by the induction assumption for $j=1$) to get that \eqref{Eq3} holds for $p_1=\cdots =p_{d-j}\in(1,\infty)$ and $p_{d-j+1}=\cdots=p_{d-1}, p_d,q\in(1,\infty)$. Again we interpolate
$p_1=\cdots =p_{d-j}\in(1,\infty)$ and $p_{d-j+1}=\cdots=p_{d-1}, p_d,q\in(1,\infty)$ with
$p_1=\cdots =p_{d-2}\in(1,\infty)$ and $p_{d-1}, p_d,q\in(1,\infty)$ (which holds by the induction assumption for $j=2$) to get that \eqref{Eq3} holds for $p_1=\cdots =p_{d-j}\in(1,\infty)$ and $p_{d-j+1}=\cdots=p_{d-2},p_{d-1}, p_d,q\in(1,\infty)$. Keep interpolating with the induction assumption for $j=3,\cdots,n-1$, we can get that \eqref{Eq3} holds for $p_1=\cdots =p_{d-j}\in(1,\infty)$ and $p_{d-j+1}, p_{d-j+2},\cdots,p_d,q\in(1,\infty)$.
		
	\vspace{1mm}\noindent
\emph{Step 5.} Finally, we proceed to prove estimate (\ref{t}).  With the same argument as in the previous 4 steps, it is sufficient to prove the following estimate:\begin{align*}
	\Vert w\Vert_{\mathbb{H}_{\alpha-2,\bf p}^\infty}^{np}\leq C_T\Vert f\Vert_{\mathbb{H}^{np}_{-2,\bf p}}^{np}, \quad {\bf p}=(p,\cdots,p,np),\quad \alpha\in[0,2-\frac{2}{np}),
	\end{align*}
where $\lim_{T\rightarrow\infty} C_T=0$.
In fact, by  \eqref{5} and \eqref{sup}, we get for any $\alpha\in[0,2-\frac{2}{np})$,
\begin{align*}
	\sup_{s\in[0,T]}\int_{\mathbb{R}}\Vert w(s,\cdot,x_d)\Vert_{H^{\alpha-2}_p(\mathbb{R}^{d-1})}^{np}\dif x_d&\leq \int_{\mathbb{R}}\sup_{s\in[0,T]}\Vert w(s,\cdot,x_d)\Vert_{H^{\alpha-2}_p(\mathbb{R}^{d-1})}^{np}\dif x_d\\&\leq C_T\Vert f\Vert_{\mathbb{H}^{np}_{-2,\bf p}}^{np}.
	\end{align*}
The whole proof can be finished.
\end{proof}

\subsection{Proof of Theorem \ref{mp}}
Now, we are in the position to give:

\begin{proof}[Proof of Theorem \ref{mp}]
	By standard continuity method, it suffices to establish the estimates (\ref{es1}) and (\ref{esss2}). Estimates (\ref{es2}) and (\ref{es3}) then follow by Sobolev embedding thoeorems,  see e.g. \cite{VN}. We divide the proof into two steps.

\vspace{1mm}\noindent	
	(i) (Case $b\equiv 0$) For ${\bf p}\in(1,\infty)^d$ and $q\in(1,\infty)$, let $u\in \sH^{q}_{2,{\bf p}}(T)$ and $f\in\mL_{\bf p}^q(T)$ satisfy \eqref{pde3}, and let $\rho_n$ be a family of standard mollifiers. Define
	$$
	u_n(t,x):=u(t,\cdot)*\rho_n(x),\ \ a_n(t,x):=a(t,\cdot)*\rho_n(x),\ \ f_n(t,x):=f(t,\cdot)*\rho_n(x).
	$$
	Then, it is easy to see that $u_n$ satisfies
	$$
	\p_t u_n=a^{ij}_n\p_{ij} u_n+g_n,\ \ u_n(0,x)=0,
	$$
	where
	$$
	g_n:=f_n+(a^{ij}\p_{ij}u)*\rho_n-a^{ij}_n\p_{ij} u_n.
	$$
	As a result of  \eqref{Eq3} and (\ref{t}), we have
	\begin{align*}
	\|\nabla^2 u_n\|_{\mL_{\bf p}^q(T)}\leq C_1\big(\|f_n\|_{\mL_{\bf p}^q(T)}+\|(a^{ij}\p_{ij}u)*\rho_n-a^{ij}_n\p_{ij} u_n\|_{\mL_{\bf p}^q(T)}\big),
	\end{align*}
	and there exists a constant $C_T$ with $\lim_{T\to\infty}C_T=\infty$ such that
	$$
	\|u_n\|_{\mH^\infty_{\alpha,{\bf p}}(T)}\leq C_T\big(\|f_n\|_{\mL_{\bf p}^q(T)}+\|(a^{ij}\p_{ij}u)*\rho_n-a^{ij}_n\p_{ij} u_n\|_{\mL_{\bf p}^q(T)}\big).
	$$
	Letting $n\to\infty$ and by the property of convolution, we can obtain the desired result.

\vspace{2mm}\noindent
	(ii)  (Case $b\neq 0$)  Let $\frac{1}{p_i}=\frac{1}{\tilde p_i}+\frac{1}{\hat p_i}$ and $\frac{1}{q}=\frac{1}{\tilde q}+\frac{1}{\hat q}$, by H\"older's inequality and Sobolev embedding theorem (see \cite{VN}), we get
	\begin{align}\label{tilde}
	\Vert b\cdot\nabla u\Vert_{\mathbb{L}_{\bf p}^q(T)}\leq C_2\Vert b\Vert_{\mathbb{L}_{\tilde {\bf p}}^{\tilde q}(T)}\Vert \nabla u\Vert_{\mathbb{L}_{\hat {\bf p}}^{\hat q}(T)}&\leq C_2\Vert b\Vert_{\mathbb{L}_{\tilde {\bf p}}^{\tilde q}(T)}\Vert u\Vert_{\mathbb{H}_{1+\theta, \bf p}^{\hat q}(T)}\nonumber\\&\leq  C_T\Vert b\Vert_{\mathbb{L}_{\tilde {\bf p}}^{\tilde q}(T)}\Vert u\Vert_{\mathbb{H}_{1+\theta, \bf p}^{\infty}(T)}.
	\end{align}
	where $\theta\in(\frac{1}{ \tilde p_1}+\cdots+\frac{1}{ \tilde p_d},1-\frac{2}{\tilde q})\subset(\frac{1}{ p_1}+\cdots+\frac{1}{ p_d},1-\frac{2}{\tilde q})$.
	By the result of (i) and  \eqref{tilde}, we have  that
	$$\Vert u\Vert_{\mathbb{H}^\infty_{1+\theta,\bf p}(T)}\leq C_T\Big(\Vert f\Vert_{\mathbb{L}_{\bf p}^q(T)}+\Vert b\cdot\nabla u\Vert_{\mathbb{L}_{\bf p}^q(T)}\Big)\leq  C_T\Big(\Vert f\Vert_{\mathbb{L}_{\bf p}^q(T)}+\Vert b\Vert_{\mathbb{L}_{\tilde {\bf p}}^{\tilde q}(T)}\Vert u\Vert_{\mathbb{H}_{1+\theta, \bf p}^{\infty}(T)}\Big).$$
	By choosing $T$ small so that $C_T\Vert b\Vert_{\mathbb{L}_{\tilde{\bf p}}^{\tilde q}(T)}<1$, we have
	\begin{align}\label{cT}
\Vert u\Vert_{\mathbb{H}^\infty_{1+\theta,\bf p}(T)} \leq C_3\Vert f\Vert_{\mathbb{L}_{\bf p}^q(T)}.
\end{align}
	Then by (i) we get
	\begin{align*}
	\|\nabla^2 u\|_{\mL_{\bf p}^q(T)}&\leq C_4\Big(\Vert f\Vert_{\mathbb{L}_{\bf p}^q(T)}+\Vert b\cdot\nabla u\Vert_{\mathbb{L}_{\bf p}^q(T)}\Big)\\
	&\leq C_4(T)\Big(\Vert f\Vert_{\mathbb{L}_{\bf p}^q(T)}+\Vert b\Vert_{\mathbb{L}_{\tilde {\bf p}}^{\tilde q}(T)}\Vert u\Vert_{\mathbb{H}_{1+\theta, \bf p}^{\infty}(T)}\Big)
	\\&\leq C_4\Vert f\Vert_{\mathbb{L}_{\bf p}^q(T)}.
	\end{align*}
Furthermore, \eqref{cT} shows that for any $\alpha\in[0,2-\frac{2}{q})$,
$$\Vert u\Vert_{\mathbb{H}^\infty_{\alpha,\bf p}(T)} \leq C_T\Vert f\Vert_{\mathbb{L}_{\bf p}^q(T)}$$
where $\lim_{T\to0} C_T=0$.
	The whole proof is finished.
\end{proof}

\section{Well-posedness of SDEs with singular coefficients}

We first provide the existence result for weak solutions of SDE (\ref{sde1}) and prove the Krylov estimate, which will play an important role below.

\begin{lemma}
	Assume {\bf (H$a$)} holds, $b\in \mathbb{L}_{\bf p}^q(T)$ with $q,p_1,\cdots,p_d\in(2,\infty]$ and $2/q+1/p_1+\cdots+1/p_d<1$. Then there exists a weak solution $(X_t)_{t\geq0}$ to  SDE (\ref{sde1}).
Moreover, for any function
$f\in \mathbb{L}_{\bf \hat p}^{\hat q}(T)$ with $\hat q, \hat p_1,\cdots,\hat p_d\in(1,\infty)$ satisfying $2/\hat q+1/\hat p_1+\cdots+1/\hat p_d<2$, we have that
	\begin{align}\label{Kry}
	\mE\left(\int_0^T|f(s,X_s)|\dif s\right)\leq C\|f\|_{\mathbb{L}_{\bf\hat p}^{\hat q}(T)},
	\end{align}
	where $C=C(d,{\bf \hat p},\hat q,\Vert b\Vert_{\mathbb{L}_{\bf p}^{q}(T)},T)$ is a positive constant.
\end{lemma}
\begin{proof}
	Firstly,  by \eqref{es2}, \eqref{es3} and following the same argument as in \cite[Theorem 2.1]{Zh1}, we can show that (\ref{Kry}) holds when $b\equiv0$. More precisely,
	for any $0<S<T<\infty$
and function
$f\in \mathbb{L}_{\bf \hat p}^{\hat q}(S,T)$ with $2/\hat q+1/\hat p_1+\cdots+1/\hat p_d<2$, there exists a constant $C(d,\hat{\bf p},\hat q)>0$ such that
	\begin{align}\label{kry}
	\mE\left(\int_S^T|f(t,Y_t)|\dif t\bigg|{\sF_S}\right)\leq C\|f\|_{\mathbb{L}_{\bf\hat p}^{\hat q}(S,T)},
	\end{align}
	where $Y_t$ solves the following SDE without drift:
	\begin{align*}
	\dif Y_t=\sigma(t,Y_t)\dif W_t,\quad Y_0=x\in\mR^d.
	\end{align*}
	Applying \eqref{kry} to $f=|b|^2$, we can get
	\begin{align*}
	\mE\left(\int_S^T|b(t,Y_t)|^2\dif t\bigg|{\sF_S}\right)\leq C\|b^2\|_{\mathbb{L}_{{\bf p}/2}^{q/2}(S,T)}=C\|b\|_{\mathbb{L}_{\bf p}^{q}(S,T)}^2.
	\end{align*}
	It then follows by Khasminskii's lemma (see \cite[Lemma 5.3]{Zh1})  that for  any constant $\kappa>0$,
	\begin{align}\label{ex}
	\mE\exp\bigg\{\kappa\int_0^T|b(s,Y_s)|^2\dif s\bigg\}\leq C(\kappa,d,{\bf p},q,\|b\|_{\mathbb{L}_{\bf p}^{q}(T)})<\infty.
	\end{align}
	As a result, we have
	\begin{align*}
	\mE\rho_T:=\mE\exp\bigg\{-\int_0^T\big[b^T\sigma^{-1}\big](s,Y_s)\dif W_s-\frac{1}{2}\int_0^T\big[b^T(\sigma\sigma^T)^{-1}b\big](s,Y_s)\dif s\bigg\}=1.
	\end{align*}
	Thus the existence of a weak solution $(X_t)_{t\geq0}$ to SDE (\ref{sde1}) follows by Girsanov's theorem. Furthermore, we can deduce that
	{\begin{align*}
	\mE\left(\int_0^T|f(s,X_s)|\dif s\right)&=\mE\left(\rho_T\int_0^T|f(t,Y_t)|\dif t\right)\\
	&\leq \left(\mE\int_0^T\!\rho_T^\alpha \dif t\right)^{1/\alpha}\left(\mE\int_S^T|f(t,Y_t)|^\beta\dif t\right)^{1/\beta},
	\end{align*}}
	where $\alpha,$ $\beta>1$ satisfying $1/\alpha+1/\beta=1$.
	Since
	\begin{align*}
	\mE\rho_T^\alpha=\mE&\bigg[\Big(\exp(-2\alpha\int_0^T [b^T\sigma]^{-1}(t,Y_t)\dif W_t-2\alpha^2\int_0^T[b^T(\sigma\sigma^T)^{-1}b](t,Y_t)\dif t)\Big)^{1/2}\\
	&\Big(\exp((4\alpha^2-\alpha)\int_0^T\big[b^T(\sigma\sigma^T)^{-1}b\big](t,Y_t)^2\dif t)\Big)^{1/2}\bigg],
	\end{align*}
	by H\"older's inequality, the fact that exponential martingale is a supermartingale,  \eqref{a} and  \eqref{ex}, we get for every $\alpha>1$, $\mE\rho_T^\alpha \leq C(\alpha,d,{\bf p},q,\|b\|_{\mathbb{L}_{\bf p}^{q}(T)})$.
	Then, it holds that
	\begin{align*}
	\mE\left(\int_0^T\!|f(t,X_t)|\dif t\right)&\leq C(\alpha,d,{\bf p},q,\|b\|_{\mathbb{L}_{\bf p}^{q}(T)},T)\left(\mE\int_0^T |f(t,Y_t)|^\beta\dif t\right)^{1/\beta}.
	\end{align*}
	Choosing $\beta$ close enough to $1$ such that  $2/\hat q+1/\hat p_1+\cdots+1/\hat p_d<2/\beta$ and taking $\bar{\bf p}=\hat{\bf p}/\beta$ , $\bar q=\hat q/\beta$ in (\ref{kry}), we can get
	\begin{align*}
	\mE\left(\int_0^T\!|f(t,X_t)|\dif t\right)&\leq  C\Vert f^\beta\Vert_{\mathbb{L}_{\bf \bar p}^{\bar q}(T)}^{1/\beta}=C\Vert f\Vert_{\mathbb{L}_{{\bf \hat p}}^{\hat q}(T)}.
	\end{align*}
	The proof is finished.
\end{proof}

Recall that  the Hardy-Littlewood maximal operator  $\cM$ is defined by
\begin{align*}
\mathcal{M}f(x):=\sup_{{\bf r}\in(0,\infty)^{d}}\frac{1}{|{\bf B_r}|}\int_{{\bf B_r}}\!f(x+y)\dif y,\quad \forall f\in L^1_{loc}(\mathbb{R}^d),
\end{align*}
where for ${\bf r}=(r_1,r_2,\cdot\cdot\cdot,r_d)$, ${\bf B_r}:=\left\{x\in\mathbb{R}^d:|x_1|<r_1,|x_2|<r_2,\cdot\cdot\cdot, |x_d|<r_d\right\}$.
For every $f\in C^\infty(\mR^d)$, it is known that there exists a constant $C_d>0$
such that for all $x,y\in \mR^d$ (see \cite[Lemma 2.1]{XXZZ}),
\begin{align}
|f(x)-f(y)|\leq C_d |x-y|\big(\cM|\nabla f|(x)+\cM|\nabla f|(y)\big),\label{ES2}
\end{align}
and the following $L^{\bf p}(\mathbb{R}^d)$-boundness of $\cM$ with ${\bf p}\in (1,\infty)^d$ holds (see \cite[Theorem 4.1]{HK}):
\begin{align}\label{GW1}
\|\cM f\|_{L^{\bf p}(\mathbb{R}^d)}\leq C_d\|f\|_{L^{\bf p}(\mathbb{R}^d)}.
\end{align}
Now, we are in the position to give:

\begin{proof}[Proof of Theorem \ref{main}]
We only need to show the pathwise uniqueness of solutions to SDE (\ref{sde1}). To this end,	we first assume that {\bf (H$a$)} holds, and  for $q\in(1,\infty)$ and ${\bf p}\in(1,\infty)^d$,
$$
|b|,|\nabla \sigma|\in L^q([0,T];L^{\bf p}(\mR^d))\quad \text{with}\quad\frac{2}{q}+\frac{1}{p_1}+\cdots+\frac{1}{p_d}<1.
$$
By Theorem \ref{mp}, there exists a function $u\in\mH_{2,{\bf p}}^q$ satisfying
	\begin{align*}
	\p_tu(t,x)+\sL^{a}_2u(t,x)+\sL^{b}_1u(t,x)+b(t,x)=0,\quad u(T,x)=0.
	\end{align*}
Define $\Phi(t,x):=x+u(t,x)$. In view of  (\ref{es3}),  we can choose $T$ small such that
	\begin{align}\label{dd}
	1/2<\|\nabla\Phi^{-1}\|_{\mL^\infty(T)}\leq 2.
	\end{align}
	Assume that  SDE (\ref{sde1}) admits two solutions $X_t^1$ and $X_t^2$. By the Krylov's estimate (\ref{Kry}), we can use It\^o's formula to get that the process $Y^i_t:=\Phi(t,X_t^i)$ satisfies
	$$
	\dif Y^i_t=\sigma(t,X^i_t)\nabla \Phi(t,X_t^i)\dif W_t=:\Psi(t,X_t^i)dW_t,\quad i=1,2.
	$$
	Let $Z_t:=X_t^1-X_t^2$, we have by (\ref{dd}) that
	\begin{align*}
	\mE|Z_t|^2\leq 2\mE|Y_t^1-Y_t^2|^2\leq2\mE\left(\int_0^t|Z_s|^2\dif A_s\right),
	\end{align*}
	where
	\begin{align*}
	A_t:=\int_0^t\frac{|\Psi(s,X_s^1)-\Psi(s,X_s^2)|^2}{|Z_s|^2}\dif s.
	\end{align*}
	Let $\rho_n$ be a family of mollifiers on $\mathbb{R}^d$, and define $\Psi^n(t,x):=\Psi(s,\cdot)\ast\rho^n(x)$.  Then we can write
	\begin{align*}
	\mE A_t&\leq \varliminf_{\epsilon\downarrow0}\mE \left(\int_0^t\frac{|\Psi(s,X_s^1)-\Psi(s,X_s^2)|^2}{|Z_s|^2}\cdot 1_{\{|Z_s|>\epsilon\}}\dif s\right)
	\\ &\leq 3\Big(\varliminf_{\epsilon\downarrow0}\lim_{n\rightarrow\infty}\mE \left(\int_0^t\frac{|\Psi^n(s,X_s^1)-\Psi(s,X_s^1)|^2}{|Z_s|^2}\cdot 1_{\{|Z_s|>\epsilon\}}\dif s\right)
	\\&\quad+\varliminf_{\epsilon\downarrow0}\lim_{n\rightarrow\infty}\mE \left(\int_0^t\frac{|\Psi^n(s,X_s^2)-\Psi(s,X_s^2)|^2}{|Z_s|^2}\cdot 1_{\{|Z_s|>\epsilon\}}\dif s\right)\\ &\quad+\varliminf_{\epsilon\downarrow0}\sup_{n\in\mathbb{N}}\mE \left(\int_0^t\frac{|\Psi^n(s,X_s^1)-\Psi^n(s,X_s^2)|^2}{|Z_s|^2}\cdot 1_{\{|Z_s|>\epsilon\}}\dif s\right)\\
	&=:3\big(\cI_1(t)+\cI_2(t)+\cI_3(t)\big).
	\end{align*}
	By  the property of mollification, it is easy to see that
	$$
	\cI_1(t)+\cI_2(t)\leq \varliminf_{\epsilon\downarrow0}\epsilon^{-2}\lim_{n\rightarrow\infty}C\Vert\Psi^n-\Psi \Vert_{\mathbb{L}^\infty(T)}^2=0.
	$$
	 As for the third term,   we can use (\ref{ES2}), the Krylov's estimate \eqref{Kry}  and (\ref{GW1}) to get that
	\begin{align*}
	\cI_3(t)&\leq C\sup_{n\in\mathbb{N}}\mE \left(\int_0^t\big[\mathcal{M}|\nabla\Psi^n|(s,X_s^1)+\mathcal{M}|\nabla\Psi^n|(s,X_s^2)|\big]^2\dif s\right)
	\\ &\leq C\sup_{n\in\mathbb{N}}\big\|\mathcal{M}|\nabla\Psi^n| \big\|_{\mathbb{L}_{\bf p}^q(T)}^2\leq C\Vert\nabla\Psi \Vert_{\mathbb{L}_{\bf p}^q(T)}^2<\infty.
	\end{align*}
	Hence, as a result of the stochastic Gronwall's inequality \cite[Lemma 3.7]{XZ2}, we can get $\mE|Z_t|^2=0$. The general case can be proved by a standard localization procedure as in \cite[Theorem 1.3]{Zh1}. The proof is finished.
\end{proof}

\end{document}